\documentclass[10pt]{amsart}
\usepackage{calrsfs}
\usepackage{soul}
\usepackage[greek,english]{babel}
\usepackage[iso-8859-7]{inputenc}
\usepackage{graphicx}
\usepackage{hyperref}
\usepackage{amsthm}
\usepackage{amssymb}
\usepackage{multirow}
\usepackage{tikz-cd}
\usepackage{amsmath}
\usepackage{todonotes}
\usepackage{amsbsy}
\usepackage[all]{xy}
\usepackage{changes}
\usepackage{enumerate}
\usepackage{comment}
\usepackage{mathtools}
\usepackage{changes}
\usepackage{mdframed}
\usepackage{faktor} 
\usepackage{stackrel}
\usepackage[shortlabels]{enumitem}

\definechangesauthor[color=orange,name={Peter Symonds}]{PS}
\definechangesauthor[color=red, name={Kostas Karagiannis}]{KK}

\usepackage[margin=1 in]{geometry}

\newtheorem{theorem}{Theorem}
\newtheorem*{theorem*}{Theorem}

\newtheorem{corollary}[theorem]{Corollary}
\newtheorem{proposition}[theorem]{Proposition}

\theoremstyle{definition}

\newtheorem{remark}[theorem]{Remark}
\newtheorem{definition}[theorem]{Definition}

\newcommand{\Z}{\mathbb{Z}}

\date{\today}

\title{Representations on canonical models of generalized Fermat curves and their syzygies}

\author[K. Karagiannis]{Kostas Karagiannis}
\address{School of Mathematics,
University of Manchester\\
Oxford Road
Manchester M13 9PL
United Kingdom}
\email{konstantinos.karagiannis@manchester.ac.uk}

\makeatletter
\newcommand{\aprod}{\mathop{\operator@font \hbox{\Large$\ast$}}}
\makeatother

\begin{document}
\begin{abstract}
We study canonical models of $\left(\mathbb{Z}/k\mathbb{Z}\right)^n$- covers of the projective line, tamely ramified at exactly $n+1$ points each of index $k$, when $k,n\geq 2$ and the characteristic of the ground field $K$ is either zero or does not divide $k$. We determine explicitly the structure of the respective homogeneous coordinate ring first as a graded $K$-algebra, next as a $\left(\mathbb{Z}/k\mathbb{Z}\right)^n$- representation over $K$, and then as a graded module over the polynomial ring; in the latter case, we give generators for its first syzygy module, which we also decompose as a direct sum of irreducible representations.
\end{abstract}

\maketitle

% \tableofcontents

%\tableofcontents
\section{Introduction}

\subsection{Generalized Humbert and Fermat curves}
At the turn of the previous century, motivated by the theory of elliptic integrals and multiply periodic functions, G. Humbert \cite{HumbertGeorges} and H.F. Baker \cite{BakerOld} independently studied  pencils of complex projective curves of genus $5$ that arise as ramified coverings of the projective line, are branched at 5 distinct points of order 2, and are stable under the action of $\left(\Z/2\Z\right)^4$. Humbert's suggestion to seek for possible generalizations, naturally lead to covers of the projective line with a fixed number of distinct branch points all of which share the same ramification index.

This was formalized in \cite{MR2391145} and \cite{Gonzalez-Diez2009-md}, where the authors considered for each pair of integers $k,n\geq 2$, complex projective curves $F_{k,n}$ that admit an action of $\left(\Z/k\Z\right)^n$ with quotient isomorphic to the projective line and branch locus consisting of $n+1$ distinct points, each of order $k$. For $(k,n)=(2,4)$ one retrieves classic Humbert curves, while for $k=2$ and arbitrary $n$ the authors call $F_{2,n}$ a generalized Humbert curve. If $k$ is allowed to vary arbitrarily and $n=2$ is fixed, one obtains plane curves branched at three distinct points, which up to a suitable M\"obius transformation can be assumed to be $\{0,1,\infty\}$. This property characterizes classic Fermat curves, which are defined by equations of the form $x^k+y^k+z^k=0$; this motivated using the term generalized Fermat curve for $F_{k,n}$.

Since their introduction 15 years ago, generalized Fermat curves have been thoroughly studied by several mathematicians, led by R. Hidalgo and his collaborators. We mention indicatively results on Fuchsian and Schottky uniformizations \cite{MR2391145}, moduli theory \cite{Gonzalez-Diez2009-md}, Jacobian varieties \cite{JacobGFK} and \cite{MR4114973}, automorphism groups \cite{HiRuKoPa}, and relations to Ihara's theory \cite{Ihara1985-it} of Braid representations of absolute Galois groups \cite{MR4186523}. Of the latest developments, we single out \cite{MR4140621}, in which, among other results, Hidalgo finds a basis for the global sections of the sheaf $\Omega_{F_{k,n}}$ of regular $1$-differentials, then proceeds to study a particular projective embedding of $F_{k,n}$. 

These two results motivated the present paper, which treats the case of higher order differentials $\Omega_{F_{k,n}}^{\otimes m}$, the induced canonical embedding of the curve, and takes into account the action of $\left(\Z/k\Z\right)^n$. The case of Fermat curves $F_{k,2}$ was recently established by the author in \cite{MR4525558}; the current work can be viewed as a simultaneous generalization of this paper for arbitrary $n$ and of Hidalgo's paper for arbitrary $m$. Apart from the particular interest to generalized Fermat curves discussed above, motivation also comes from the broader theme of homogeneous coordinate rings, which we briefly present in the next subsection.

\subsection{Homogeneous coordinate rings}
Let $K$ be an algebraically closed field of arbitrary characteristic. Consider a triple $(X,G,\mathcal{L})$, where $X$ is a smooth projective curve over $K$, $G$ is a finite subgroup of ${\rm Aut}_k\left(X\right)$ and $\mathcal{L}$ is a very ample $G$-equivariant line bundle. 
The reader should always keep in mind the triple $(F_{k,n},\left(\Z/k\Z\right)^n,\Omega_{F_{k,n}})$ discussed in the previous section as a toy example for what will follow.

Assuming that $X$ is not hyperelliptic and has genus $g$ which is at least $3$, one has a projective embedding $X\hookrightarrow \mathbb{P}_k\left(H^0\left(X,\mathcal L\right)\right)$, which by construction must be $G$-equivariant. The homogeneous coordinate ring of $X$ relative to this embedding, denoted by $S_{X,\mathcal{L}}$, can be endowed with several algebraic structures, which reflect the geometric and representation-theoretic properties of the triple $(X,G,\mathcal{L})$.
\begin{enumerate}
\item It carries the structure of a graded $K$-algebra, since $S_{X,\mathcal{L}}=\bigoplus_{m\geq 0}H^0\left(X,\mathcal{L}^{\otimes m}\right)$.
\item The action of $G$ extends to linear representations on each $H^0\left(X,\mathcal{L}^{\otimes m}\right)$, which respect the grading, thus making $S_{X,\mathcal{L}}$ into a graded $KG$-module.
\item The projective embedding induces a graded ring homomomorphism from $S={\rm Sym}\left(H^0\left(X,\mathcal L\right)\right)$ to $ S_{X,\mathcal{L}}$, endowing the latter with the structure of a graded module over the former.
\item The actions of $G$ on $S$ and $S_{X,\mathcal{L}}$ satisfy the necessary compatibility properties so that  $S_{X,\mathcal{L}}$ becomes a graded $SG$-module.
\end{enumerate}

The primary goal of this paper is to demonstrate that one can study all four structures simultaneously, starting with (1) and working their way up into (4) as follows. First, construct explicit bases over $K$ for the graded pieces $H^0\left(X,\mathcal{L}^{\otimes m}\right)$. Proceed with decomposing $S_{X,\mathcal{L}}$ as a direct sum of indecomposable $KG$-modules, graded-piece-by-graded-piece. Then describe the kernel of the structure map $S\rightarrow S_{X,\mathcal{L}}$ and resolve $S_{X,\mathcal{L}}$ to get its graded syzygies and Betti numbers. Finally decompose the syzygy modules into direct sums of indecomposables.

It is worth mentioning here that the problem of determining each structure has been extensively studied in the literature in varying levels of generality, and has motivated much of the interaction between arithmetic geometry and representation theory in the recent decades. The bibliography is vast, and we restrict to mentioning a small subset of the relevant results. Determining (2) is referred to as the Galois module structure problem for sheaf cohomology, which can be traced back to Hecke \cite{MR3069500}. The case ${\rm char}(K)\nmid |G|$ was settled in \cite{MR589254}, generalizing work of Chevalley and Weil \cite{MR3069638}; however, the formulas are given in terms of the local monodromy at ramification points which is often hard to compute, as demonstrated in our work \cite{MR4501574}. The case of wild ramification remains open and there exist only partial results, see for example \cite{Koeck:04}, \cite{MR2434287}, \cite{karan} and \cite{MR4130074}. The problem falls into the  general context of equivariant Euler characteristics \cite{MR1274097}. Historically, the starting point for (3) is K. Petri's analysis \cite{MR1512232} of the classic theorem by M. Noether and Enriques \cite{Saint-Donat73}. Generating sets of different flavors for the first syzygy module have been computed; the approach taken here builds on previous work of the author \cite{charalambous2019relative}, closely related to that of \cite{MR3337885} and \cite{MR4302549}. The most explicit result on higher syzygies is Schreyer's algorithm \cite{MR1129577}, and much more is known about the Betti table of $S_{X,\Omega_X}$, see \cite[\S 9]{MR2103875} and \cite{MR3729076}. The de facto central problem of the area has been M. Green's syzygy conjecture, settled by Voisin in characteristic $0$ in \cite{MR1941089}, \cite{MR2157134}, then generalized in \cite{MR4022070}; the topic remains an active area of research. Group actions on resolutions of graded modules over polynomial rings, the general framework for (4), have been studied in relation to Boij-S\"oderberg theory \cite{MR2918721}, Veronese subalgebras \cite{almousa2022equivariant}, and permutations of monomial ideals \cite{MR4407114}, \cite{MR4155201}. For an overview of computational flavor, see \cite{MR4381679}. To the author's knowledge, the case of homogeneous coordinate rings has been minimally explored. The only relevant reference is \cite{MR4333646}.

In conclusion, the different structures have been mostly treated separately so far and it is our hope to unify these themes under the same framework and explore possible connections. The author believes that this can also shed light to  the generalized version of Oort's conjecture on lifting curves with automorphisms \cite{Obus2017}, settled for cyclic groups in \cite{ObusWewers} and \cite{MR3194816}, but still open in its full generality.
\subsection{Outline}
In Section \ref{sec:preliminaries} we review the basic properties of generalized Fermat curves from \cite{MR4140621}. The main ingredients are the presentation of their automorphisms in eq. (\ref{eq:aut}), their description as fiber products of classic Fermat curves in eq. (\ref{eq:model2}) and the basis for the global sections of regular 1-differentials in eq. (\ref{eq:1-diff}). We then proceed with our study of differentials of higher order. First in Section \ref{sec:diffs} as vector spaces over the ground field (Theorem \ref{th:basis}) then as representations (Theorem \ref{th:mult}). In Section \ref{sec:canonical} we consider their direct sum as a graded module over a polynomial ring by invoking the classic result of M. Noether and Enriques. By using a variant of the techniques of \cite{charalambous2019relative} and \cite{MR3337885}, we prove in Theorem \ref{th:main} that the first syzygy module is generated by a collection of binomials and trinomials, defined in eq. (\ref{eq:binoms}) and  eq. (\ref{eq:trinoms}) respectively. Finally, we extend the group action both on the said polynomial ring and on the syzygy module and decompose them into direct sums of irreducible representations in Proposition \ref{prop:decomp-sym} and Corollary \ref{cor:decomp-syz} respectively.

\subsection*{Acknowledgments} The author would like to thank Ioannis Tsouknidas for bringing Hidalgo's paper to their attention and for suggesting to apply the techniques of \cite{charalambous2019relative} to prove Theorem \ref{th:main}, and Aristides Kontogeorgis for helpful comments on early versions of this paper. This research was supported by EPSRC grant no. EP/V036017/1.

\section{Preliminaries on Generalized Fermat Curves}\label{sec:preliminaries}

Let $K$ be an algebraically closed field of characteristic $p\geq 0$, and let $k,n\geq 2$ be integers such that $(k-1)(n-1)>1$ and $p\nmid k$. \emph{A generalized Fermat curve of type }$(k, n)$ is a non-singular projective algebraic curve $F_{k,n}$ over $K$
admitting a group of automorphisms $H$ isomorphic to $\left(\mathbb{Z}/k\mathbb{Z}\right)^{n}$ such that the quotient $F_{k,n}/H$ is the projective line $\mathbb{P}^1_K$ with $n + 1$ distinct branch points, each one of order $k$. For $n=2$, one retrieves the definition of the classic Fermat curve, given by the vanishing of $X_0^k+X_1^k+X_2^k=0$ in $\mathbb{P}^2_K$.

To make things more explicit, consider the complete intersection in $\mathbb{P}^n_K$ given by the fiber product of $(n-1)$-many classic Fermat curves
\begin{equation*}\label{eq:model}
C_{\lambda_1,\ldots,\lambda_{n-1}}^k:  \left\{\lambda_iX_0^k+X_1^k+X_{i+1}^k=0:1\leq i\leq n-1  \right\}\subset \mathbb{P}_K^n,
\end{equation*}
where $\lambda_1=1$ and $\lambda_i\in K-\{0,1,\infty\}\text{ are pairwise distinct for }2\leq i\leq n-1$. Each classic Fermat curve above is branched at $0,\infty$ and $\lambda_i$ for $2\leq i\leq n$; thus, by setting $\lambda_0=0$ and $\lambda_{n}=\infty$, the branch locus $\{\lambda_i:0\leq i\leq n\}$ of $C_{\lambda_1,\ldots,\lambda_{n-1}}^k$ consists of $n + 1$ distinct points. Further, if $\zeta=\zeta_k$ is a primitive $k$-th root of unity and $H$ is the group generated by the automorphisms
\begin{equation}\label{eq:aut}
\{\phi_{j}:0\leq j\leq n\} \text{ where } \phi_jX_i=\zeta^{\delta_{ij}}X_i,
\end{equation}
 then one has that $\phi_0\cdots\phi_n=1$ and so $H\cong \left(\mathbb{Z}/k\mathbb{Z}\right)^{n}$. Thus, $C_{\lambda_1,\ldots,\lambda_{n-1}}^k$ is a generalized Fermat curve of type $(k, n)$; its genus can be directly computed from the Riemann-Hurwitz formula
\begin{equation}\label{genus}
g_{k,n} = 
|H|\left(g_{\mathbb{P}^1_K}-1\right)+1+\sum_{i=0}^n\left(1-\frac{1}{k}\right)=
1 + \frac{k^{n-1}}{2}\left[(k-1)(n-1)-2\right].
\end{equation}
The converse statement, that any generalized Fermat curve of type $(k, n)$ arises in the above manner, was first proved for $K=\mathbb{C}$ in \cite[Theorem 5]{Gonzalez-Diez2009-md}, then extended to arbitrary $K$ in \cite[Theorem 2.2]{MR4140621}. In what follows, we identify $F_{k,n}$ with the complete intersection $C_{\lambda_1,\ldots,\lambda_{n-1}}^k$. We shall be working with the affine model obtained by setting $x=X_1/X_0$ and $y_i=X_i/X_0$ for $2\leq i\leq n$,
\begin{equation}\label{eq:model2}
 F_{k,n}:
 \left\{
  \begin{array}{rl}
  1+x^k+y_2^k &=0\\
  \lambda_2+x^k+y_3^k &=0\\
&\vdots\\
  \lambda_{n-1}+x^k+y_{n}^k &=0
  \end{array}
\right\}.
\end{equation}
Let $\Omega_{F_{k,n}}$ denote the canonical sheaf on $F_{k,n}$. In \cite[Theorem 3.1]{MR4140621}, the author proves that the collection
\begin{equation}\label{eq:1-diff}
\mathcal{B}_{k,n}
=
\left\{
\frac{x^r dx}{y_2^{a_2}y_3^{a_3}\dots y_{n}^{a_{n}}}:
0\leq a_i\leq k-1\text{ for all }2\leq i\leq n,\text{ and }0\leq r\leq \sum_{i=2}^{n}a_i-2
\right\}
\end{equation}
forms a $K$-basis for vector the space of global sections $H^0\left(F_{k,n},\Omega_{F_{k,n}}\right)$. Assuming that $(n-1)(k-1)>2$, one has by \cite[Theorem 4]{Gonzalez-Diez2009-md} that $F_{k,n}$ is not hyperelliptic and thus, the choice of basis $\mathcal{B}_{k,n}$ above gives rise to an embedding $F_{k,n}\hookrightarrow\mathbb{P}^{g_{k,n}-1}_K$. These results constitute the starting point of our analysis.

\section{Holomorphic $m$-differentials}\label{sec:diffs}

For $m\geq 1$, let $\Omega_{F_{k,n}}^{\otimes m}$ be the sheaf of regular $m$-differentials on $F_{k,n}$. The global sections $V_m=H^0(F_{k,n},\Omega_{F_{k,n}}^{\otimes m})$ form a vector space over the ground field $K$ of dimension
\begin{equation}\label{eq:hilb}
d_m=\dim_K V_m
=
\begin{cases}
g_{k,n}&, \text { if }m=1 \\
(2m-1)(g_{k,n}-1)&, \text{ if }m\geq 2.
\end{cases}
\end{equation}
To describe an explicit basis for each $V_m$, consider the {\em meromorphic} differentials
\begin{equation}\label{def:theta}
\theta_{r,\mathbf{a}}^{(m)}=\frac{x^r dx^{\otimes m}}{y_2^{a_2}y_3^{a_3}\dots y_{n}^{a_{n}}},\;\text{where }
(r,\mathbf{a})=(r,a_2,\ldots,a_n)\in\mathbb{Z}^n.
\end{equation}
The divisors ${\rm div}(\theta_{r,\mathbf{a}}^{(m)})$ can be computed using a similar argument to that used in \cite[\S 3.2]{MR4140621}: for each $0\leq j\leq n$, let $\phi_j$ be the automorphism defined in eq. (\ref{eq:aut}), let $\{P_{j\ell}:1\leq \ell\leq k^{n-1}\}$ denote its set of fixed points and set $D_j=\sum_{\ell=0}^{k^{n-1}}P_{j,\ell}$ to be the corresponding divisor. We then have that
\[
{\rm div}(x)=-D_0+ D_1+\sum_{i=2}^n D_i, \quad
{\rm div}(y_j)=-D_0+D_j,\quad
{\rm div}(dx)=-2D_0+\sum_{i=2}^n(k-1) D_i,
\]
and thus, denoting the sum $\sum_{i=2}^{n}a_i$ by $|\mathbf{a}|$, we get
\begin{equation}\label{eq:div}
{\rm div} (\theta_{r,\mathbf{a}}^{\otimes m})=
\left(|\mathbf{a}|-2m-r\right)D_0
+ rD_1
+\sum_{i=2}^{n}[m(k-1)-a_i]D_i.
\end{equation}
We are now ready to prove the following.

\begin{theorem}\label{th:basis}
For $m\geq 1$, let
\begin{equation*}\label{def:Im}
I_{k,n}^{(m)}=
\left\{
(r,\mathbf{a})\in\mathbb{Z}^{n}:
(m-1)(k-1)\leq a_i\leq m(k-1)\text{ for all }2\leq i\leq n,
\text{ and }
0\leq r\leq |\mathbf{a}|-2m
\right\}.
\end{equation*}
Then the collection
\[
\mathcal{B}_{k,n}^{(m)}=
\left\{
\theta_{r,\mathbf{a}}^{(m)}=\frac{x^r dx^{\otimes m}}{y_2^{a_2}y_3^{a_3}\dots y_{n}^{a_{n}}}:
(r,\mathbf{a})\in I_{k,n}^{(m)}
\right\}
\]
is a basis for $V_m=H^0\left(F_{k,n},\Omega_{F_{k,n}}^{\otimes m}\right)$.
\end{theorem}
\begin{proof}
By eq. (\ref{eq:div}), the differentials in $\mathcal{B}_{k,n}^{(m)}$ are holomorphic. It thus suffices to show that the cardinality of $I_{k,n}^{(m)}$ equals $d_m$, as given in eq. (\ref{eq:hilb}). We proceed by induction on $m$.

For $m=1$ we retrieve the collection $\mathcal{B}_{k,n}$ of eq. (\ref{eq:1-diff}) and the result follows from \cite[Theorem 3.1]{MR4140621}. Assume that $|I_{k,n}^{(m)}|=(2m-1)(g_{k,n}-1)$ and consider the set
\[
I_{k,n}^{(m+1)}=
\left\{
(r,\mathbf{a})\in\mathbb{Z}^{n}:
m(k-1)\leq a_j\leq (m+1)(k-1)\text{ and }
0\leq r\leq |\mathbf{a}|-2m-2
\right\}.
\]
For $2\leq j\leq n$ set $b_j=a_j-(k-1)$, so that $|\mathbf{b}|=|\mathbf{a}|-(n-1)(k-1)$. Then
\[
I_{k,n}^{(m+1)}=
\left\{
(r,\mathbf{b})\in\mathbb{Z}^{n}:
(m-1)(k-1)\leq b_j\leq m(k-1)\text{ and }
0\leq r\leq |\mathbf{b}|+(n-1)(k-1)-2m-2
\right\}
\]
can be written as the disjoint union of the two sets 
\[
\left\{(r,\mathbf{b})\in\mathbb{Z}^{n}:
(m-1)(k-1)\leq b_j\leq m(k-1)\text{ and }
0\leq r\leq |\mathbf{b}|-2m
\right\}
\]
and 
\[
\left\{
(r,\mathbf{b})\in\mathbb{Z}^{n}:
(m-1)(k-1)\leq b_j\leq m(k-1)\text{ and }
|\mathbf{b}|-2m+1\leq r\leq |\mathbf{b}|+(n-1)(k-1)-2m-2
\right\}.
\]
The former equals $I_{k,n}^{(m)}$, so it has cardinality equal to $(2m-1)(g_{k,n}-1)$ by the inductive hypothesis. For the cardinality of the latter, note that the inequalities $(m-1)(k-1)\leq b_j\leq m(k-1)$ are satisfied by $k^{n-1}$ tuples $\mathbf{b}$, and that each such tuple
gives rise to exactly
$
(n-1)(k-1)-2
$
values of $r$. Thus
\[
|I_{k,n}^{(m+1)}|
=(2m-1)(g_{k,n}-1)+k^{n-1}\left[(k-1)(n-1)-2\right]
=(2m-1)(g_{k,n}-1)+2(g_{k,n}-1)=(2m+1)(g_{k,n}-1)
\]
completing the proof.
\end{proof}
We proceed with studying the action of $H\cong \left(\mathbb{Z}/k\mathbb{Z}\right)^n$ on $V_m$. To be consistent with the indexing notation introduced in eq. (\ref{def:theta}), we denote the elements of $\left(\mathbb{Z}/k\mathbb{Z}\right)^n$ by $(e_1,{\bf e})$, where ${\bf e}=(e_2,\ldots,e_n)$, and the corresponding elements of $H$ by $\sigma_{(e_1,{\bf e})}$. Recall that $H$ acts on $F_{k,n}$ via
$
\sigma_{(e_1,{\bf e})}(x,y_2,\ldots,y_n)=(\zeta^{e_1}x,\zeta^{e_2}y_2,\ldots,\zeta^{e_n}y_n),
$ 
and this induces an action on the basis $\mathcal{B}_{k,n}^{(m)}$ of $V_m$ via
\begin{equation}\label{eq:diffaction}
\sigma_{(e_1,{\bf e})}\theta_{r,\mathbf{a}}^{(m)}
=\frac{\left(\zeta^{e_1}x\right)^r d\left(\zeta^{e_1}x\right)^{\otimes m}}{\left(\zeta^{e_2}y_2\right)^{a_2}\left(\zeta^{e_3}y_3\right)^{a_3}\dots \left(\zeta^{e_n}y_{n}\right)^{a_{n}}}
=
\zeta^{e_1(r+m)-\sum_{j=2}^na_je_j}\theta_{r,\mathbf{a}}^{(m)}
=
\zeta^{e_1(r+m)-{\bf a \cdot e}}\theta_{r,\mathbf{a}}^{(m)},
\end{equation}
where we write ${\bf a \cdot e}$ for the sum $\sum_{j=2}^na_je_j$. Thus, the character of this representation is
\begin{equation}\label{eq:chars}
\chi_{V_m}:H\rightarrow K,\;\sigma_{(e_1,{\bf e})}\mapsto\sum_{(r,\mathbf{a})\in I_{k,n}^{(m)}}\zeta^{e_1(r+m)-{\bf a \cdot e}}.
\end{equation}
We identify the irreducible representations of $H$ with its group of irreducible characters
\[
\mathcal{X}(H)=
\left\{
\chi_{(h_1,{\bf h})}:H\rightarrow K,\;\sigma_{(e_1,{\bf e})}\mapsto
\zeta^{h_1e_1+\bf h\cdot e}\mid (h_1,{\bf h})
\in \left(\mathbb{Z}/k\mathbb{Z}\right)^n 
\right\},
\]
and write $
V_m=\displaystyle\bigoplus_{(h_1,{\bf h})
\in \left(\mathbb{Z}/k\mathbb{Z}\right)^n }\nu_{m,(h_1,{\bf h})}\chi_{(h_1,{\bf h})}
$
for the decomposition of $V_m$ into a direct sum of irreducibles.
\begin{theorem}\label{th:mult}
$\displaystyle
\nu_{m,(h_1,{\bf h})}
=
(n-1)(m-1)-\left\lceil\frac{|(h_1,{\bf h})|+m}{k}\right\rceil-\sum_{j=2}^{n}\left\lfloor \frac{m-1-h_j}{k} \right\rfloor+1.
$
\end{theorem}
\begin{proof}
Using eq. (\ref{eq:chars}) and the fact that $\nu_{m,(h_1,{\bf h})}=\langle \chi_{V_m},\chi_{(h_1,{\bf h})}\rangle$, we obtain
\begin{equation*}
\nu_{m,(h_1,{\bf h})}=\frac{1}{|H|}\sum_{(e_1,{\bf e})\in \left(\mathbb{Z}/k\mathbb{Z}\right)^n}
\chi_{V_m}\left(\sigma_{(e_1,{\bf e})}\right)\overline{\chi_{(h_1,{\bf h})}\left(\sigma_{(e_1,{\bf e})}\right)}
=
\frac{1}{k^n}\sum_{(e_1,{\bf e})\in \left(\mathbb{Z}/k\mathbb{Z}\right)^n \atop (r,\mathbf{a})\in I_{k,n}^{(m)}}\zeta^{e_1(r+m-h_1)-({\bf a}+{\bf h}) \cdot{\bf e}}.
\end{equation*}
For a fixed value $(r,\mathbf{a})\in I_{k,n}^{(m)}$ we have
\[
\sum_{(e_1,{\bf e})\in \left(\mathbb{Z}/k\mathbb{Z}\right)^n }\zeta^{e_1(r+m-h_1)-({\bf a}+{\bf h}) \cdot{\bf e}}
=
\begin{cases}
k^n&, \text { if } \left(r+m-h_1, {\bf a}+{\bf h}\right)\equiv {\bf 0}\;{\rm mod\;}k \\
0&, \text{ otherwise,}
\end{cases}
\]
and thus $\nu_{m,(h_1,{\bf h})}$ is equal to the cardinality of  $
\left\{
(r,\mathbf{a})\in I_{k,n}^{(m)}: \left(r+m-h_1, {\bf a}+{\bf h}\right)\equiv {\bf 0}\;{\rm mod\;}k
\right\}
$.
Spelled out, one needs to count the number of tuples $(r,\mathbf{a})\in \mathbb{Z}^n$ that satisfy the relations
\[
\left\{
  \begin{array}{r}
  (m-1)(k-1)\leq a_j\leq m(k-1)\\
  0\leq r\leq |\mathbf{a}|-2m
    \end{array}
\text{ and }
\begin{array}{r}
   r+m\equiv h_1\;{\rm mod\;}k\\
  a_j\equiv -h_j\;{\rm mod\;}k
  \end{array}
\right\},
\]
or equivalently, setting $b_j=a_j-(m-1)(k-1)$, the relations
\[
\left\{
  \begin{array}{r}
  0\leq b_j\leq k-1\\
  0\leq r\leq |\mathbf{b}|+(n-1)(m-1)(k-1)-2m
    \end{array}
\text{ and }
\begin{array}{r}
    r+m\equiv h_1\;{\rm mod\;}k\\
  b_j\equiv m-1-h_j\;{\rm mod\;}k.
  \end{array}
\right\}.
\]
There is exactly one tuple ${\bf b}$ that satisfies $0\leq b_j\leq k-1$ and $b_j\equiv m-1-h_j\;{\rm mod\;}k$, namely
\[
b_j=m-1-h_j-\left\lfloor\frac{m-1-h_j}{k}\right\rfloor k.
\]
Write $r+m=qk+h_1$ for the division of $r+m$ by $k$ and substitute into the inequality to be satisfied by $r$

\begin{align*}
0\leq qk+h_1-m&\leq (n-1)(m-1) -|{\bf h}|-\sum_{j=2}^{n}\left\lfloor \frac{m-1-h_j}{k} \right\rfloor k +(n-1)(m-1)(k-1)-2m\\
\Leftrightarrow 
0\leq qk&\leq 
 (n-1)(m-1)k-|(h_1,{\bf h})|-m-\sum_{j=2}^{n}\left\lfloor \frac{m-1-h_j}{k} \right\rfloor k\\
\Leftrightarrow 
0\leq  q&\leq 
(n-1)(m-1)-\left\lceil\frac{|(h_1,{\bf h})|+m}{k}\right\rceil-\sum_{j=2}^{n}\left\lfloor \frac{m-1-h_j}{k} \right\rfloor
\end{align*}
and the result follows.
\end{proof}

\section{The canonical ring}\label{sec:canonical}
Let $S_{F_{k,n}}$ be the direct sum of the $K$-vector spaces $V_m=H^0(\Omega_{F_{k,n}}^{\otimes m},F_{k,n})$ for $m\geq 0$. It is equipped with the structure of a graded ring, with multiplication defined as 
\[
V_m\times V_{m'}\rightarrow V_{m+m'},\; 
\left(fdx^{\otimes m},gdx^{\otimes m'}\right)\mapsto fgdx^{\otimes (m+m')}.
\]  
Let $S$ be the symmetric algebra ${\rm Sym}\left(V_1\right)$; using Theorem \ref{th:basis}, we identify $S$ with the polynomial ring with variables $\{z_{r,\bf{a}}:(r,{\bf a})\in I_{k,n}^{(1)}\}$, indexed by the elements of 
\[
I_{k,n}^{(1)}
=
\left\{
(r,{\bf a})\in\mathbb{Z}^n:0\leq a_j\leq k-1\text{ and }0\leq r\leq |a|-2
\right\}.
\]
The assignment $z_{r,\bf{a}}\mapsto\theta_{r,\bf{a}}^{(1)}$ can be naturally extended to a homogeneous homomorphism of graded rings
$\phi:S\rightarrow S_{F_{k,n}}$, which endows $S_{F_{k,n}}$ with the structure of a graded $S$-module. Assuming that $(n-1)(k-1)>2$, one has by \cite[Theorem 4]{Gonzalez-Diez2009-md} that $F_{k,n}$ is not hyperelliptic. Thus, we may invoke the following classic result of M. Noether, Enriques and Petri, see\cite{Saint-Donat73}.
\begin{theorem}\label{th:Petri}
The canonical map $\phi:S\rightarrow S_{F_{k,n}}$ is surjective. If $F_{k,n}$ is neither trigonal nor a plane quintic, then the kernel of $\phi$ is generated by homogeneous elements of degree $2$.
\end{theorem}

In what follows, we assume that $F_{k,n}$ is neither trigonal nor a plane quintic and proceed with the description of a generating set for $\ker\phi$ in degree $2$, by considering the $K$-linear map 
\[
\phi_2:S_2\twoheadrightarrow V_2,\;
z_{r,\bf{a}}z_{s,\bf{b}}\mapsto\theta_{r+s,\bf{a+b}}^{(2)}.
\]
Notice that $\phi\left(z_{r,\bf{a}}z_{s,\bf{b}}\right)=\phi\left(z_{t,\bf{c}}z_{u,\bf{d}}\right)\Leftrightarrow (r+s,{\bf a}+{\bf b})=(t+u,\bf{c}+{\bf d})$ and so
\begin{equation}\label{eq:binoms}
\mathcal{G}_{\rm bi}
=
\left\{
z_{r,\bf{a}}z_{s,\bf{b}}-z_{t,\bf{c}}z_{u,\bf{d}}:
(r+s,{\bf a}+{\bf b})=(t+u,\bf{c}+{\bf d})
\right\}\subseteq\ker\phi.
\end{equation}
\begin{remark}\label{rem:points}
To give a combinatorial interpretation of the above, recall that in degree $1$ we have a bijection between the variables of $S$, the basis of $V_1$ and the points of $I_{k,n}^{(1)}$ . In degree $2$, there is a similar correspondence, which however fails to be a bijection. By Theorem \ref{th:basis}, the basis of $V_2$ is in bijection with
\[
I_{k,n}^{(2)}
=
\left\{
(r,{\bf a})\in\mathbb{Z}^n:k-1\leq a_j\leq 2(k-1)\text{ and }0\leq r\leq |a|-4
\right\}.
\]
The images of the monomials in $S_2$ under $\phi$ correspond to points in the Minkowski sum
\begin{eqnarray*}
I_{k,n}^{(1)}+I_{k,n}^{(1)}
=
\left\{
(r,{\bf a})+(s,{\bf b}):(r,{\bf a}),(s,{\bf b})\in I_{k,n}^{(1)}
\right\}=
\left\{
(r,{\bf a})\in\mathbb{Z}^n:0\leq a_j\leq 2(k-1)\text{ and }0\leq r\leq |a|-4
\right\},
\end{eqnarray*}
which strictly contains $I_{k,n}^{(2)}$. Two monomials correspond to the same point if and only they give rise to a binomial in $\mathcal{G}_{\rm bi}$. 
\end{remark}

To obtain the remaining generators, for $1\leq i\leq n-1$ let ${\bf k}(i)$ be the vector $(0,\ldots,k,\ldots,0)\in \mathbb{Z}^{n-1}$ whose $i$-th coordinate equals $k$ and all others are $0$. Let $(r,{\bf a})$ be a point in $I_{k,n}^{(1)}+I_{k,n}^{(1)}$ such that $(r+k,{\bf a})\in I_{k,n}^{(1)}+I_{k,n}^{(1)}$ and $(r,{\bf a-k}(i))\in I_{k,n}^{(1)}+I_{k,n}^{(1)}$ for exactly one $1\leq i\leq n-1$. These three points give rise to a relation in $V_2$
\[
\lambda_i \theta_{r,{\bf a}}^{(2)}+\theta_{r+k,{\bf a}}^{(2)}+\theta_{r,{\bf a-k}(i)}^{(2)}
=\theta_{r,{\bf a}}^{(2)}\left(\lambda_i+x^k+y^k\right)
\]
which is zero by eq. (\ref{eq:model2}). The preimage of that relation under $\phi$ defines a subset of $\ker\phi$, and so
\begin{equation}\label{eq:trinoms}
\mathcal{G}_{\rm tri}
=
\big\{
\lambda_i
z_{r,\bf{a}}z_{s,\bf{b}}+
z_{t,\bf{c}}z_{u,\bf{d}}+
z_{v,\bf{e}}z_{w,\bf{f}}
:
\begin{array}{l}
t+u=r+s+k,\; {\bf  c+d}={\bf  a+b}\\
v+w=r+s,\;  {\bf e+f}={\bf a+b-k}(i)
\end{array},\;
1\leq i\leq n-1
\big\}
\subseteq \ker\phi.
\end{equation}
%The image of an element of $G_{\rm trinom}$ under the canonical map is
%\begin{eqnarray*}
%\lambda_i
%\theta_{r+s,\bf{a+b}}+
%\theta_{t+u,\bf{c+d}}+
%\theta_{v+w,\bf{e+f}}
%&=&
%\lambda_i
%\theta_{r+s,\bf{a+b}}+
%\theta_{r+s+k,\bf{a+b}}+
%\theta_{r+s,\bf{a+b-k}(i)}\\
%&=&
%\lambda_i
%\frac{x^{r+s}dx^{\otimes 2}}{y_2^{a_2+b_2}\cdots y_{n}^{a_{n}+b_n}}+
%\frac{x^{r+s+k} dx^{\otimes 2}}{y_2^{a_2+b_2}\cdots y_{n}^{a_{n}+b_n}}+
%\frac{x^{r+s} dx^{\otimes 2}}{y_2^{a_2+b_2}\cdots y_i^{a_i+b_i-k}\cdots y_{n}^{a_{n}+b_n}}\\
%&=&
%\frac{x^{r+s}dx^{\otimes 2}}{y_2^{a_2+b_2}\cdots y_{n}^{a_{n}+b_n}}
%\left(
%\lambda_i+x^k+y_{i+1}^k\right)
%=
%\theta_{r+s,\bf{a+b}}
%\left(
%\lambda_i+x^k+y_{i+1}^k\right)=0,
%\end{eqnarray*}
\begin{theorem}\label{th:main}
$\ker\phi$ is generated by $ \mathcal{G}_{\rm bi}\cup \mathcal{G}_{\rm trinom}$.
\end{theorem}
The proof of Theorem \ref{th:main} will be given using a variant of the techniques introduced in \cite{charalambous2019relative}, \cite{MR3337885} and \cite{MR4302549}. To this end, we recall some facts of Gr\"obner theory from \cite[\S 1]{MR1363949}. Let $\prec$ be a term order on the monomials of $S$. For $f\in S$, let ${\rm in}_\prec(f)$ be the initial term of $f$ with respect to $\prec$; for an ideal $\mathfrak{a}$ of $S$ and a subset $\mathcal{G}\subseteq \mathfrak{a}$, let ${\rm in}_\prec\left(\frak{a}\right)=\langle {\rm in}_\prec(f): f\in\mathfrak{a}\rangle$ and ${\rm in}_\prec\left(\mathcal{G}\right)=\{{\rm in}_\prec(f):f\in \mathcal{G}\}$. In general $\langle {\rm in}_\prec\left(\mathcal{G}\right) \rangle \subseteq {\rm in}_\prec\left(\langle \mathcal{G}\rangle\right)$, with equality if and only if $\mathcal{G}$ is a Gr\"obner basis for $\langle \mathcal{G}\rangle$. A monomial is called {\em standard} with respect to $\mathfrak{a}$ if it does not lie in ${\rm in}_\prec\left(\frak{a}\right)$ and standard monomials form a $K$-basis for $S/\mathfrak{a}$. In the case of interest to this paper, when $\mathfrak{a}=\ker\phi$ and $\mathcal{G}=\mathcal{G}_{\rm bi}\cup \mathcal{G}_{\rm tri}$, we will identify the standard monomials with a subset of $I_{k,n}^{(1)}+I_{k,n}^{(1)}$, in the philosophy of Remark \ref{rem:points}. For convenience purposes we fix a term order, even though the proof is independent of the choice, see also Remark \ref{rem:choice}. In the definition below, we extend the notation ${\bf a}=(a_2,\ldots,a_n)\in\mathbb{Z}^{n-1}$ to account for collections of elements of $\mathbb{Z}^{n-1}$ which will be denoted by ${\bf a}_1=\left(a_{1,2},\ldots,a_{1,n}\right),\;{\bf a}_2=\left(a_{2,2},\ldots,a_{2,n}\right)$ and so on.
\begin{definition}\label{def:order}
Let $\prec$ be the term order on the monomials of $S$ defined as:
\[
z_{r_1,{\bf a}_1}
z_{r_2,{\bf a}_{2}}
\cdots 
z_{r_d,{\bf a}_{d}}
\prec
z_{s_1,{\bf b}_{1}}
z_{s_2,{\bf b}_{2}}
\cdots 
z_{s_2,{\bf b}_{d'}}
\text{ if and only if }
\] 
\begin{itemize}
\item $d<d'$ or 
\item $d=d'$ and $\sum r_i> \sum s_i$ or 
\item $d=d'$ and $\sum r_i= \sum s_i$ and $\sum_i a_{i,2}<\sum_ib_{i,2}$ or
\item $d=d'$ and $\sum r_i= \sum s_i$ and $\sum_i a_{i,2}=\sum_ib_{i,2}$ and $\sum_i a_{i,3}<\sum_ib_{i,3}$ or
\item[$\vdots$]
\item $d=d'$ and $\sum r_i= \sum s_i$ and $\sum_i a_{i,j}=\sum_ib_{i,j}$ for all $2\leq j\leq n$ and
\[
z_{r_1,{\bf a}_1}
z_{r_2,{\bf a}_{2}}
\cdots 
z_{r_d,{\bf a}_{d}}
<
z_{s_1,{\bf b}_{1}}
z_{s_2,{\bf b}_{2}}
\cdots 
z_{s_2,{\bf b}_{d'}}
\text{ lexicographically.}
\] 
\end{itemize}
\end{definition}
\begin{proof}[Proof of Theorem \ref{th:main}]
By construction, $\mathcal{G}=\mathcal{G}_{\rm bi}\cup\mathcal{G}_{\rm tri}\subseteq \ker\phi$, see equations (\ref{eq:binoms}) and (\ref{eq:trinoms}). By Theorem \ref{th:Petri}, the inclusion $\ker\phi\subseteq\langle \mathcal{G}\rangle$ can be checked in degree $2$; we will show that $\left( S/\langle \mathcal{G}\rangle\right)_2$ is a $K$-subspace of $\left( S/\ker\phi\right)_2$. By Gr\"obner theory, referring again indicatively to \cite[\S 1]{MR1363949}, we have that
\begin{equation}\label{eq:ineq2}
\dim_K\left( S/\langle \mathcal{G}\rangle\right)_2
=
\dim_K\left( S/ {\rm in}_\prec\langle\mathcal{G}\rangle\right)_2
\leq
\dim_K\left( S/ \langle{\rm in}_\prec\left(\mathcal{G}\right)\rangle\right)_2
=
|\mathbb{T}^2\setminus {\rm in}_\prec\left(\mathcal{G}\right)|,
\end{equation}
where $\mathbb{T}^2$ denotes the set of monomials of degree $2$ in $S$. To obtain an upper bound for $|\mathbb{T}^2\setminus {\rm in}_\prec\left(\mathcal{G}\right)|$, we consider the map of sets
\[
\tau: I_{k,n}^{(1)}+I_{k,n}^{(1)}\rightarrow \mathbb{T}^2,\;
(r,{\bf a})\mapsto\underset{\prec}{{\rm min}}
\left\{
z_{s,{\bf b}}z_{t,{\bf c}}\in\mathbb{T}^2:(s+t,{\bf b+c})=(r,{\bf a})
\right\},
\]
which is well-defined, 1-1 and has image equal to $\mathbb{T}^2\setminus{\rm in}_\prec\left(\mathcal{G}_{\rm bi}\right)$. Thus $\mathbb{T}^2\setminus{\rm in}_\prec\left(\mathcal{G}_{\rm bi}\right)$ is in bijection with $ I_{k,n}^{(1)}+I_{k,n}^{(1)}$. For $2\leq i\leq n$, consider the subsets of $ I_{k,n}^{(1)}+I_{k,n}^{(1)}$ defined as
\begin{equation}\label{eq:Ci}
C_i
=
\left\{
(r,{\bf a})\in I_{k,n}^{(1)}+I_{k,n}^{(1)}:
(r+k,{\bf a})\in I_{k,n}^{(1)}+I_{k,n}^{(1)}\text{ and } (r,{\bf a-k}(i))\in I_{k,n}^{(1)}+I_{k,n}^{(1)}
\right\}.
\end{equation}
The three monomials $\tau\left(r,{\bf a}\right),\;\tau\left(r+k,{\bf a}\right)$ and $\tau\left(r,{\bf a-k}(i)\right)$ give rise to an element of $\mathcal{G}_{\rm tri}$
\[
\lambda_i\tau\left(r,{\bf a}\right)+\tau\left(r+k,{\bf a}\right)+\tau\left(r,{\bf a-k}(i)\right),
\]
whose initial term is by construction $\tau\left(r,{\bf a}\right)$. Thus $|\bigcup_{i=2}^n C_i|\leq |{\rm in}_\prec\left(\mathcal{G}_{\rm tri}\right)|$, and so
\begin{equation*}\label{eq:ineq}
|\mathbb{T}^2\setminus {\rm in}_\prec\left(\mathcal{G}\right) |\leq
|(I_{k,n}^{(1)}+I_{k,n}^{(1)})\setminus \bigcup_{i=2}^n C_i|
=
|(I_{k,n}^{(1)}+I_{k,n}^{(1)})\bigcap_{i=2}^n {\overline C_i}|.
\end{equation*}
Using the defining inequalities of $I_{k,n}^{(1)}+I_{k,n}^{(1)}$, we observe that
\[
C_i
=
\left\{
(r,{\bf a})\in\mathbb{Z}^n:
0\leq a_j\leq 2k-2\text{ for }j\neq i,\;
k\leq a_i\leq 2k-2\text{ and }0\leq r\leq |{\bf a}|-(k+4)
\right\}.
\]
Setting $b_i=a_i-k$ and $b_j=a_j$ for $i\neq j$ yields
\[
C_i=
\left\{
(r,{\bf b})\in\mathbb{Z}^n:
0\leq b_j\leq 2k-2\text{ for }j\neq i,\;
0\leq b_i\leq k-2\text{ and }0\leq r\leq |{\bf b}|-4
\right\}
\]
and so
\[
\left(I_{k,n}^{(1)}+I_{k,n}^{(1)}\right)\bigcap_{i=2}^n {\overline C_i}
=
\left\{
(r,{\bf b})\in\mathbb{Z}^n:
k-1\leq b_i\leq 2(k-1)\text{ for }
2\leq i\leq n-1
\text{ and }0\leq r\leq |{\bf b}|-4
\right\}=I_{k,n}^{(2)}.
\]
Thus $|\mathbb{T}^2\setminus{\rm in}_\prec\left(\mathcal{G}\right)|\leq |I_{k,n}^{(2)}|$.
Theorem \ref{th:basis} and  the first part of Theorem \ref{th:Petri} then imply that
\[
|\mathbb{T}^2\setminus {\rm in}_\prec\left(\mathcal{G}\right)|\leq   |I_{k,n}^{(2)}|
 =
 \dim_kV_2=\dim_K\left( S/\ker\phi\right)_2,
\]
and so by eq. (\ref{eq:ineq2}), $\left( S/\langle \mathcal{G}\rangle\right)_2$ is a $K$-subspace of $\left( S/\ker\phi\right)_2$, as requested.
\end{proof}

\begin{remark}\label{rem:choice}
The term order of Definition \ref{def:order} was only used to argue that $\tau$ maps the sets $C_i$ of eq. (\ref{eq:Ci}) into ${\rm in}_\prec\left(\mathcal{G}_{\rm tri}\right)$, which follows since the monomial $\tau\left(r,{\bf a}\right)$ is greater than both $\tau\left(r+k,{\bf a}\right)$ and $\tau\left(r,{\bf a-k}(i)\right)$ with respect to $\prec$. Had we refrained from making a choice, we could have given a case-by-case proof: if for example $\tau\left(r+k,{\bf a}\right)$ were maximal among the three monomials we would replace $C_i$ by the set
\[
\left\{
(r,{\bf a})\in I_{k,n}^{(1)}+I_{k,n}^{(1)}:
(r-k,{\bf a})\in I_{k,n}^{(1)}+I_{k,n}^{(1)}\text{ and } (r-k,{\bf a-k}(i))\in I_{k,n}^{(1)}+I_{k,n}^{(1)}
\right\}
\] 
which has the same cardinality as $C_i$, and similarly for the third case.
\end{remark}
The action of $H$ on $V_1$ gives rise to a natural action on $S={\rm Sym}(V_1)$ which respects the grading: if $\mathfrak{m}$ is a monomial of degree $d$, and $\sigma_{(e_1,{\bf e})}\in H$ is the automorphism determined by $(e_1,{\bf e})\in\left(\Z/k\Z\right)^n$, then $\sigma_{(e_1,{\bf e})}\mathfrak{m}=\zeta^{e_1(r+d)-{\bf a\cdot e}}\mathfrak{m}$, where $(r,{\bf a})$ is the sum of the indices corresponding to the variables that divide $\mathfrak{m}$. 

This leads us to consider, for each $d\geq 1$, the $d$-fold Minkowski sum of $I_{k,n}^{(1)}$ with itself 
\[
d\cdot I_{k,n}^{(1)}=\underbrace{I_{k,n}^{(1)}+I_{k,n}^{(1)}+\cdots +I_{k,n}^{(1)}}_{d\text{-times}}.
\]
For each point $(r,{\bf a})\in d\cdot I_{k,n}^{(1)}$, let $S_{d,r,{\bf a}}$ be the $k$-span of
\[
\left\{
z_{r_1,{\bf a}_1}
z_{r_2,{\bf a}_{2}}
\cdots 
z_{r_d,{\bf a}_{d}}\in S_d
\;|\
\sum_{i=1}^d\left( r_i,{\bf a}_i\right)=(r,{\bf a})
\right\},
\]
i.e. the set of monomials $\mathfrak{m}$ of degree $d$ such that the sum of the indices corresponding to the variables that divide $\mathfrak{m}$ equals $(r,{\bf a})$. The above construction gives rise to a direct sum decomposition
\begin{equation}\label{eq:mhilb}
S
=\bigoplus_{d=0}^{\infty}S_d
=\bigoplus_{d=0}^{\infty}\bigoplus_{(r,{\bf a})\in d\cdot I_{k,n}^{(1)}} S_{d,r,{\bf a}},
\end{equation}
which determines the structure of $S$ as a $KG$-module.
\begin{proposition}\label{prop:decomp-sym}
The multiplicity $\mu_{d,(h_1,{\bf h})}$ of the irreducible representation $\chi_{(h_1,{\bf h})}$ of $H$ in the decomposition  of $S_d$ in a direct sum of irreducible representations is given by
\[
\mu_{d,(h_1,{\bf h})}=
\sum_{(r,{\bf a})\in J^{(d)}_{h_1,{\bf h}}}
\dim_K\left( S_{d,r,{\bf a}}\right),\text{ where }
J^{(d)}_{h_1,{\bf h}}=
\left\{
(r,{\bf a})\in d\cdot I_{k,n}^{(1)}\;\mid\; (r+d,-{\bf a})\equiv(h_1,{\bf h})\;{\rm mod}\;k
\right\}.
\]
\end{proposition}
\begin{proof}
An automorphism $\sigma_{(e_1,{\bf e})}\in H$ acts on the elements of $S_{d,r,{\bf a}}$ by the scalar $\zeta^{e_1(r+d)-{\bf a\cdot e}}$
and thus $S_{d,r,{\bf a}}$ is isomorphic, as a $KG$-module, to a direct sum of $\dim_K\left( S_{d,r,{\bf a}}\right)$ copies of the irreducible representation determined by the class of $(r+d,-{\bf a})\in\left(\mathbb{Z}/k\mathbb{Z}\right)^n$. Hence,
\[
S_d=\bigoplus_{(r,{\bf a})\in d\cdot I_{k,n}^{(1)}} \dim_K\left( S_{d,r,{\bf a}}\right)\chi_{(r+d,-{\bf a})}
\]
and the result follows since two points $(r,{\bf a}),(s,{\bf b})\in d\cdot I_{k,n}^{(1)}$ determine the same irreducible representation if and only if $(r,{\bf a})\equiv(s,{\bf b})\;{\rm mod}\;k$. 
\end{proof}

Finally, we extend the action of $H$ on the ideal $\ker\phi$ of Theorem \ref{th:Petri}. This can be done either by comparing the action of $H$ on $S$ to that on $S_{F_{k,n}}$ and deduce that the canonical map $\phi$ is $H$-equivariant, or by directly verifying that for any generator $f\in\mathcal{G}_{\rm bi}\cup\mathcal{G}_{\rm tri}$ we have $\sigma_{(e_1,{\bf e})}f=\zeta^{e_1(r+2)-{\bf a\cdot e}}f$. Further, the assumption that ${\rm char}(K)\nmid |H|$ allows us to invoke Maschke's theorem to conclude that the sequence $0\rightarrow\ker\phi\rightarrow S\rightarrow S_{F_{k,n}}\rightarrow 0$ of Theorem \ref{th:Petri} is a split short exact sequence of $KG$-modules.
\begin{corollary}\label{cor:decomp-syz}
The multiplicity of the irreducible representation $\chi_{(h_1,{\bf h})}$ of $H$ in the decomposition  of $\left(\ker\phi\right)_d$ in a direct sum of irreducible representations is given by $\mu_{d,(h_1,{\bf h})}-\nu_{d,(h_1,{\bf h})}$, where $\mu_{d,(h_1,{\bf h})}$ is an in Proposition \ref{prop:decomp-sym} and $\nu_{d,(h_1,{\bf h})}$ is as in Theorem \ref{th:mult}.
\end{corollary}

We conclude our analysis with some comments on the relationship of our techniques with algebraic combinatorics and combinatorial commutative algebra. This has a twofold purpose: first to present some future directions and open problems, and second to justify why  
the formulas for $\mu_{d,(h_1,{\bf h})}$ in Proposition \ref{prop:decomp-sym} are not as concrete as the respective formulas for $\nu_{d,(h_1,{\bf h})}$ given in Theorem \ref{th:mult}.
\begin{enumerate}
\item For $d=1,\;\dim_K\left( S_{1,r,{\bf a}}\right)=1$ for all $(r,{\bf a})\in I_{k,n}^{(1)}$, and so $\mu_{1,(h_1,{\bf h})}=\nu_{1,(h_1,{\bf h})}$, reflecting the fact that $\ker\phi$ has no generators in degree 1. For $d\geq 1,\;\dim_K S_{d,r,{\bf a}}$ can be interpreted as the number of partitions of the vector $(r,{\bf a})\in\mathbb{Z}^n$ into $d$-many parts, such that each summand lies in $I_{k,n}^{(1)}$. The enumeration of such {\em multipartite partitions} is a classic, albeit difficult, problem in algebraic combinatorics, usually approached via the theory of generating functions, see \cite{MR1357776}. The requirement that the summands must be in $I_{k,n}^{(1)}$ poses an additional obstruction to obtaining explicit formulas.
\item The standard generating functions used in invariant theory of polynomial rings are equivariant Hilbert-Poincar\'e series, see for example \cite{MR526968} or \cite[\S 3]{MR1869812}. The most well-known and concrete result is Molien's formula, which expresses the generating function of the sequence $\{\mu_{d,(h_1,{\bf h})} \}_{d=0}^\infty$ as a sum of rational functions. Applying this to our case, even when $\chi_{(h_1,{\bf h})}$ is the trivial representation, does not produce a formula more illuminating than the ones obtained in Proposition \ref{prop:decomp-sym}.
\item For yet another class of relevant formal power series, we mention Erhart theory \cite[\S 12]{MR2110098}, which addresses the problem of enumerating lattice points in dilations of convex polytopes. In our context, these dilations correspond to the sets $d\cdot I_{k,n}^{(1)}$ but the desired stratification by the sets $J_{h_1,{\bf h}}^{(d)}$ in Proposition \ref{prop:decomp-sym} would require a version of Erhart theory modulo $k$. It is also worth mentioning that there is a connection with the so-called integer decomposition property of polytopes, see \cite{MR3292265}.
\end{enumerate}
In any case, the potential interplay between the three types of generating functions in the context of the present paper is something that we deem worth exploring in future work.
\begin{enumerate}
\item[(4)] The indexing of the indeterminates $z_{r,{\bf a}}$ of the polynomial ring $S$ by the elements of $I_{k,n}^{(1)}$ can be interpreted as a multigrading of $S$ by the abelian group $\mathbb{Z}^{n+1}$: the multidegree of a monomial $\mathfrak{m}=z_{r_1,{\bf a}_1}\cdots z_{r_d,{\bf a}_d}$ is ${\rm mdeg}\left(\mathfrak{m}\right)=\left(d,\sum r_i,\sum {\bf a}_i\right)\in \mathbb{Z}\times d\cdot I_{k,n}^{(1)}\subset\mathbb{Z}^{n+1}$. The decomposition of eq. (\ref{eq:mhilb}) thus gives rise to a multigraded Hilbert function via the assignment $(d,r,{\bf a})\mapsto\dim_kS_{d,r,{\bf a}}$, which is a refinement of the classic Hilbert function $d\mapsto \binom{d+g-1}{d}$ of the polynomial ring. A similar interpretation can be given for the multiplicities of the irreducible summands of the degree $d$ pieces of  $S,\;S_{F_{k,n}}$ and $\ker\phi$ as $KH$-modules. This leads towards the theory of invariant and multigraded Hilbert schemes, see \cite{MR3184162} and \cite{MR2073194} respectively; a natural question would be to seek for defining equations for the Hilbert schemes parametrizing ideals with the above multigraded Hilbert functions.
\end{enumerate}

%\begin{eqnarray*}
%\sigma_{(e_1,{\bf e})}
%z_{r_1,{\bf a}_1}
%z_{r_2,{\bf a}_{2}}
%\cdots 
%z_{r_m,{\bf a}_{m}}
%&=&
%\zeta^{e_1(r_1+1)-{\bf a_1\cdot e}}z_{r_1,{\bf a}_1}
%\zeta^{e_1(r_2+1)-{\bf a_2\cdot e}}z_{r_2,{\bf a}_2}
%\cdots
%\zeta^{e_1(r_m+1)-{\bf a_d\cdot e}}z_{r_d,{\bf a}_m}\\
%&=&
%\zeta^{e_1\left(\sum_{i=1}^m r_i +m\right)+{\bf e\cdot}\sum_{i=1}^m{\bf a}_i}
%z_{r_1,{\bf a}_1}
%z_{r_2,{\bf a}_{2}}
%\cdots 
%z_{r_m,{\bf a}_{m}}
%\end{eqnarray*}

 \def\cprime{$'$}

 \bibliographystyle{alpha}
\bibliography{KKGeneral.bib}

\end{document}